\documentclass[a4paper]{scrartcl}

\usepackage[english]{babel}
\usepackage{amsmath,amsfonts,amssymb,amsthm}
\usepackage[nocompress,noadjust]{cite}
\usepackage{url}
\usepackage[colorlinks]{hyperref}
\usepackage{algorithm}
\usepackage[noend]{algpseudocode}
\usepackage{calc}
\usepackage{tikz}
\usetikzlibrary{patterns}

\newtheorem{defin}{Definition}
\newtheorem{lemma}[defin]{Lemma}

\newtheorem{remark}[defin]{Remark}
\newtheorem{proposition}[defin]{Proposition}

\newcommand{\rk}{\text{rk}}
\newcommand{\F}{\mathbb{F}}

\newcommand{\gaussmset}[2]{\left[\begin{smallmatrix}{#1}\\{#2}\end{smallmatrix}\right]}
\newcommand{\gaussmnum}[3]{\gaussmset{#1}{#2}_{#3}}
\newcommand{\gaussmsetminusset}[3]{\gaussmset{{#1}{\setminus}{#2}}{#3}}
\newcommand{\gaussmsetminusnum}[4]{\gaussmsetminusset{#1}{#2}{#3}_{#4}}

\newcommand{\twolines}[2]{#1 #2}

\title{New LMRD bounds for constant dimension codes and improved constructions}
\author{Daniel~Heinlein\thanks{
The author is with the Department of Mathematics, Physics, and Computer Science, University of Bayreuth, Bayreuth, GERMANY. Email: firstname.lastname@uni-bayreuth.de\newline{}The work was supported by the ICT COST Action IC1104 and grants KU 2430/3-1, WA 1666/9-1 -- ``Integer Linear Programming Models for Subspace Codes and Finite Geometry'' -- from the German Research Foundation.}}

\begin{document}

\maketitle

\begin{abstract}
We generalize upper bounds for constant dimension codes containing a lifted maximum rank distance code first studied by Etzion and Silberstein.
The proof allows to construct several improved codes.

\smallskip
\noindent \textbf{Keywords:}
Finite projective spaces,
constant dimension codes,
subspace codes,
subspace distance,
rank distance,
maximum rank distance codes,
lifted maximum rank distance bound,
combinatorics.
\end{abstract}

\section{Introduction}\label{sec:introduction}

Let $V \cong \mathbb{F}_q^v$ be a $v$-dimensional vector space over the finite field $\mathbb{F}_q$ with $q$ elements.
By $\gaussmset{V}{k}$ we denote the set of all $k$-dimensional subspaces in $V$.
Its size is given by the $q$-binomial coefficient $\gaussmnum{v}{k}{q}=\prod_{i=0}^{k-1}\frac{q^v-q^i}{q^k-q^i}$ for $0 \le k \le v$ and $0$ otherwise.

The set of all subspaces of $V$ forms a metric space associated with the so-called subspace distance $d_S(U,W) = \dim(U+W) - \dim(U \cap W)$, cf. \cite[Lemma~1]{MR2451015}.
A $(v,M,d;k)_q$ constant dimension code (CDC) $C$ is a subset of $\gaussmset{V}{k}$ of cardinality $M$ in which for each pair of elements, called codewords, the subspace distance is lower bounded by $d$, i.e., we have $d \le d_S(U,W)$ for all $U \ne W \in C$.

The main question of subspace coding in the constant dimension case asks for the maximum cardinality $M$ for fixed parameters $q$, $v$, $d$, and $k$ of a $(v,M,d;k)_q$ code.
The maximum cardinality is denoted as $A_q(v,d;k)$.

$A_q(v,d;k)$ is known for some parameters.
By definition, $A_q(v,d;k)=0$ for $k<0$ or $v<k$.
If $d \le 2$, then $A_q(v,d;k)=\gaussmnum{v}{k}{q}$.
Let $U^\perp$ denote the orthogonal complement of $U$ with respect to a fixed non-degenerate symmetric bilinear form on $V$.
Since $d_S(U^\perp, W^\perp)=d_S(U,W)$, we have $A_q(v,d;k) = A_q(v,d;v-k)$, cf. \cite[Remark after Lemma~1]{MR2597176}, and hence may assume $k \le v/2$.
If $2k<d$, any code has at most one element.
The subspace distance in the constant dimension case is always even: $d_S(U,W)=2(k-\dim(U \cap W))$ for $U,W \in \gaussmset{V}{k}$.
Therefore we occasionally use the assumption $2 \le d/2 \le k \le v/2$.

Note that for $U\ne W$ in a $(v,\#C,d;k)_q$ CDC $C$ the subspace distance yields $\dim(U \cap W) \le k-d/2$.
Therefore any at least $(k-d/2+1)$-dimensional subspace of $V$ is contained in at most one codeword.

A prominent code construction uses maximum rank distance (MRD) codes.
A linear rank metric code $[m \times n, M, d]_q$ is a subspace $C$ of the vector space of $m \times n$ matrices over $\mathbb{F}_q$, i.e., $\mathbb{F}_q^{m \times n}$, of cardinality $M$, for which the distance of two elements is lower bounded via the rank metric $d_r(A,B) = \rk(A-B)$, i.e., $d \le d_r(A,B)$ for all $A \ne B \in C$. For all parameters, $0 \le m, n, d$ and $q$ prime power, there is a linear rank metric code that attains the maximum cardinality of $\left\lceil q^{\max\{m,n\}(\min\{m,n\}-d+1)} \right\rceil$, cf. \cite{MR791529}.

The lifted MRD (LMRD) code \cite[Proposition~4]{MR2450762} is a $(v,\#M,d;k)_q$ CDC $C$ that uses a $k \times k$ identity matrix $I_k$ as prefix for a $[k \times (v-k),\#M,d/2]_q$ MRD code $M$, where $2 \le d/2 \le k \le v/2$ implies $\#M=q^{(v-k)(k-d/2+1)}$: $C=\{ \operatorname{rowspan}(I_k \mid A) : A \in M \}$.
The horizontal concatenation of matrices, having the same number of rows, is denoted by ``$\mid$''.

The arising question of upper bounds on sizes for CDCs that contain an LMRD as subset was partly answered by Etzion and Silberstein in \cite[Theorem~10 and Theorem~11]{MR3015712}.
This paper generalizes both bounds in Proposition~\ref{prop:1} and Proposition~\ref{prop:2} such that both bounds together cover the parameter range $k < 3d/2$ together with $2 \le d/2 \le k \le v/2$.

Since the writing of \cite{MR3015712} there are quite a few works that can profit of a generalized LMRD bound.
First of all Etzion asked in Research Problem~5 of his survey of open problems~\cite{etzion2013problems} and the authors of \cite{heinlein2017coset} asked in the conclusion for a generalization of the LMRD bound.
Next the expurgation-augmentation method of Honold et al. \cite{liu2014new,ai2016expurgation} often surpasses the LMRD bound and is therefore stronger than all constructions that include an LMRD as subset.
The homepage \url{http://subspacecodes.uni-bayreuth.de} bundled with the manual in \cite{HKKW2016Tables} lists some explicit calculations of lower and upper bounds and particularly the LMRD bound for small parameters.
Finally, there are multiple papers that use the LMRD bound and can profit of this generalization \cite{MR3440233,MR3329980,silberstein2013new,MR3367813,MR3705116,heinlein2017new}.

The main result of this paper is summarized in this proposition.
\begin{proposition}\label{prop:0}
For $2 \le d/2 \le k \le v/2$ let $C$ be a $(v,\#C,d;k)_q$ CDC that contains an LMRD code.

If $k<d \le 2/3 \cdot v$ we have
\[\#C \le q^{(v-k)(k-d/2+1)} + A_q(v-k,2(d-k);d/2)\text{.}\]
If additionally $d=2k$, $r \equiv v \mod{k}$, $0 \le r <k$, and $\gaussmnum{r}{1}{q}<k$, then the right hand side is equal to $A_q(v,d;k)$ and achievable in all cases.

If $(v, d, k) \in \{ (6+3l,4+2l,3+l), (6l,4l,3l) \mid l \ge 1 \}$, then there is a CDC containing an LMRD with these parameters whose cardinality achieves the bound.

If $k<d$ and $v<3d/2$ we have
\[\#C \le q^{(v-k)(k-d/2+1)} + 1\]
and this cardinality is achieved.

If $d \le k < 3d/2$ we have
\begin{align*}
\#C
&
\le
q^{(v-k)(k-d/2+1)} + A_q(v-k,3d-2k;d)
\\
&
+
\gaussmnum{v-k}{d/2}{q}\gaussmnum{k}{d-1}{q}
q^{(k-d+1)(v-k-d/2)}
/
\gaussmnum{k-d/2}{d/2-1}{q}
\text{.}
\end{align*}
\end{proposition}

For fixed $q$ and $v$,
Figure~\ref{fig} visualizes the parameter regions of $d$ and $k$ in which which if clause of Proposition~\ref{prop:0} is applicable.
The style is based on the tables in \url{http://subspacecodes.uni-bayreuth.de} \cite{HKKW2016Tables}.

\begin{figure}
\centering
\begin{tikzpicture}[xscale=0.3,yscale=0.3]
\node[above] at (2,-2) {$k=d=2$};
\draw[->] (2,-2) -- (10,-2) node[right] {$k$};
\draw[->] (2,-2) -- (2,-20) node[below] {$d$};
\draw (2,-4) -- (10,-20) node[right] {$d=2k$};
\draw (2,-2) -- (10,-10) node[right] {$d=k$};
\draw (3,-2) -- (10,-6-2/3) node[right] {$k=3d/2$};
\draw (6+2/3,-13-1/3) -- (10,-13-1/3) node[right] {$v=3d/2$};
\fill[pattern=vertical lines] (10,-2) -- (3,-2) -- (10,-6-2/3);
\fill[pattern=horizontal lines] (10,-10-0.4) -- (2,-2-0.4) -- (2,-4) -- (6+2/3,-13-1/3) -- (10,-13-1/3);
\fill[pattern=north east lines] (10,-6-2/3-0.4) -- (3,-2-0.4) -- (2,-2) -- (10,-10);
\fill[pattern=dots] (10,-13-1/3) -- (6+2/3,-13-1/3) -- (10,-20);

\end{tikzpicture}
\caption{
In analogy to the tables in \url{http://subspacecodes.uni-bayreuth.de}, see also \cite{HKKW2016Tables},
for fixed $q$ and $v$ the image 
shows the general knowledge about LMRD bounds.
From top to bottom: For parameters in the area with vertical lines no LMRD bound is known, then Proposition~\ref{prop:2} is the best LMRD bound, then Proposition~\ref{prop:1} is the best LMRD bound, and in the dotted area the LMRD bound is trivial.
}
\label{fig}
\end{figure}
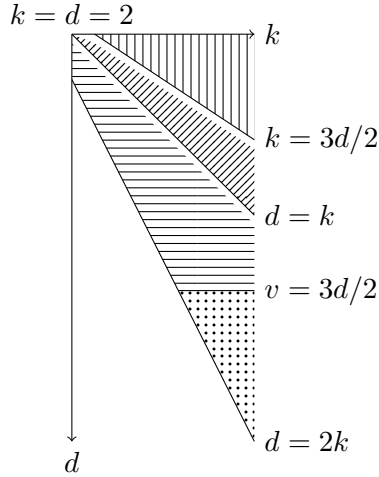

The paper is organized as follows.
We collect basic facts and definitions about constant dimension codes in Section~\ref{sec:introduction} and Section~\ref{sec:preliminaries}.
The two main bounds are proved in Section~\ref{sec:maintheorems}.
Since the second bound depends on two parameters we show how to choose these parameters to get the strongest bound in Section~\ref{sec:comparisonbounds}.
In this section, we also compare the second with the first bound.
The proof of Proposition~\ref{prop:0} is then presented, together with final remarks, in Section~\ref{sec:proof_prop0}.
Section~\ref{sec:code_improvements} constructs an addendum for an LMRD along the proof of Proposition~\ref{prop:1}, effectively increasing some lower bounds.
A conclusion is drawn in Section~\ref{sec:conclusion}.

\section{Preliminaries}\label{sec:preliminaries}

In the remainder of the paper, we need well-known facts about $q$-binomial coefficients.

Using $[x]_q=\gaussmnum{x}{1}{q}=(q^x-1)/(q-1)$ for integral $x \ge 0$ and $q$-factorials $[x]_q!=\prod_{i=1}^{x} [i]_q$, the $q$-binomial coefficient is $\gaussmnum{v}{k}{q}=\frac{[v]_q!}{[k]_q! [v-k]_q!}$ for $0 \le k \le v$ and $0$ otherwise.
We will use the inequalities $\frac{a-1}{b-1} \le (\ge) \frac{a}{b}$ if $1 < b$ and $a \le (\ge) b$ and $\frac{[x]_q}{[y]_q} \le (\ge) q^{x-y}$ if $0 < y$ and $x \le (\ge) y$.

Let $\mu(q)=\prod_{i=1}^{\infty}(1-q^{-i})^{-1}$, then $q^{k(v-k)} \le \gaussmnum{v}{k}{q} \le \mu(q) q^{k(v-k)}$ \cite[Lemma~4]{MR2451015}.
Note that $\mu(q)$ is monotonically decreasing in $q$ and some approximated values are given in Table~\ref{tab:mu}.

\begin{table}
\centering
\caption{Values for $\mu(q)$ and $\log_q(\mu(q))$ for small $q$.}\label{tab:mu}
\begin{tabular}{l|lllllll}
$q$              & 2    & 3    & 4    & 5    & 7    & 8    & 9    \\
\hline
$\mu(q)$         & 3.46 & 1.79 & 1.45 & 1.32 & 1.20 & 1.16 & 1.14 \\
$\log_q(\mu(q))$ & 1.79 & 0.53 & 0.27 & 0.17 & 0.09 & 0.07 & 0.06 \\
\end{tabular}
\end{table}

Moreover, one of the two Pascal identities for $q$-binomial coefficients is
$\gaussmnum{v}{k}{q}=\gaussmnum{v-1}{k}{q}q^k+\gaussmnum{v-1}{k-1}{q}$.

An upper bound for the size of CDCs is the Singleton bound:
\begin{lemma}[{\cite[Theorem~9]{MR2451015}}]\label{lem:singleton_bound}
For $q \ge 2$ prime power, $v$, $d/2$, $k$ integers with $d/2 \le \min\{k,v-k\}$:
\[ A_q(v,d;k)\!\le\!\gaussmnum{v-d/2+1}{\max\{k,v-k\}}{q}\!\!\!\!=\!\min\left\{\!\gaussmnum{v-d/2+1}{k}{q}\!\!\!\!, \gaussmnum{v-d/2+1}{v-k}{q}\!\right\}\!\text{.}\]
\end{lemma}

Successive zeros and ones are abbreviated:

$1_l=\underbrace{1 \ldots 1}_{l}$ and $0_l=\underbrace{0 \ldots 0}_{l}$.

The bijection $\tau$ between a Grassmannian and an appropriate set of full-rank matrices in reduced row echelon form (RREF)
\[
\tau_{q,v,k}:\gaussmset{\mathbb{F}_q^v}{k} \rightarrow \{A \in \mathbb{F}_q^{k \times v} \mid \rk(A)=k, \text{$A$ is in RREF}\}
\]
will also be applied multiple times. If $q$, $v$, and $k$ is clear from the context, we will abbreviate $\tau_{q,k,v}$ with $\tau$.

By 
\[\Gamma_{q,k,v} = \tau^{-1}(0_{(v-k) \times k} \mid I_{v-k})\]
we denote the $(v-k)$-dimensional subspace of $V$ that contains all vectors which start with $k$ zeros.
We use this to partition the vector space
\[V = \  \Gamma_{q,k,v} \ \dot\cup \  \Delta_{q,k,v}\text{,}\]
hence $\Delta_{q,k,v}$ contains all $q^v-q^{v-k}$ vectors of $V$ whose first $k$ entries are not $0_k$ each.

Note that the authors of \cite{MR3329980} denote $\Gamma_{q,k,v}$ special flat and that we again drop the reference to $q$, $v$, and $k$ if it is clear from the context.

A $k$-spread $S$ in $V$ is a subset of $\gaussmset{V}{k}$ such that all nonzero vectors of $V$ are partitioned in subspaces in $S$, hence $S$ is a $(v,(q^v-1)/(q^k-1),2k;k)_q$ CDC.
It exists iff $k \mid v$ \cite{MR0169117}.

A partial $k$-spread $P$ in $V$ is a subset of $\gaussmset{V}{k}$ such that some nonzero vectors of $V$ are packed in subspaces in $S$, hence it is a $(v,\#P,2k;k)_q$ CDC.
The question of the maximum cardinality $\#P$ is not setteled, cf. \cite{MR0404010,MR560442,MR2576869,MR3631662,MR3682737,MR3682916,MR0169117}.
Quite recently, it could be answered for many parameters.
\begin{lemma}[{\cite[Theorem~5]{MR3682737}}]\label{lem_size_ps}
For $r \equiv v \mod{k}$, $0 \le r <k \le v/2$, and $[r]_q<k$: $A_q(v,2k;k)=(q^v-q^{k+r})/(q^k-1)+1$.
\end{lemma}
With the exception of $21$ sporadic cases in \cite{MR3631662}, \cite[\S VI]{MR0169117} and \cite[Theorem~2.9 and~2.10]{MR3631662} describe the strongest upper bounds for partial spreads, the latter can be derived by interpreting the set of non-covered $1$-dimensional subspaces as columns of a generator matrix of a linear code, cf. \cite{heinlein2017projective,honold2016partial}.

A lower bound for CDCs, which in particular meets the upper bound in Lemma~\ref{lem_size_ps}, is given by the Echelon-Ferrers construction~\cite{MR2589964}. Its main ingredient is the following lemma which connects the subspace distance to the Hamming distance using the pivot vector $p_{q,v,k}:\gaussmset{V}{k} \rightarrow \F_2^v$ such that $p_{q,v,k}(U)_i=1$ iff the $i$-th columnn of $\tau(U)$ is a pivot column for $U \le V$. If the context implies $q$, $v$, and $k$, we abbreviate $p_{q,v,k}$ with $p$.
\begin{lemma}[{\cite[Lemma~2]{MR2589964}}]\label{lem:dh_ds}
If $U,W \le V$ then $d_S(U,W) \ge d_H(p(U),p(W))$.
\end{lemma}
For each codeword $c$ of a binary constant weight code of length $v$, weight $k$, i.e., each non-zero codeword has exactly $k$ ones, and Hamming distance $d$, the Echelon-Ferrers construction builds a CDC $C_c$ using codewords $M$ of a $[k \times (v-k),N,d/2]_q$ rank metric code with prescribed zeros such that $M$ \emph{fits} in a RREF matrix with pivots in the positions of the ones of $c$.
The final CDC is the the union of each $C_c$.

Although it is an open question how this rank metric code may be constructed in the general case, for the scope of this paper we only need:
\begin{lemma}[{cf. \cite[Theorem~9]{MR3480069}}]\label{lem:EFMRD_special_case}
Let $A$ be a $[a \times a',l,d_a]_q$ and $B$ a $[b \times b',l,d_b]_q$ rank metric code.
Then there is a $[(a+b) \times (a'+b'),l,d_a+d_b]_q$ rank metric code such that each codeword contains a zero matrix of size $b \times a'$ in the bottom left corner.
\end{lemma}

$\mathcal{S}_n$ is the symmetric group of permutations of $n$ elements.

$\mathcal{H}_k(U)$ is an arbitrary $k$-dimensional subspace of a vector space $U$, cf. \cite[before Definition~1]{MR2451015}.

Moreover, we need to count the number of subspaces which lie in a given subspace and only intersect another given subspace trivially.

\begin{defin}
Let $W$ and $U$ be subspaces of $V$. The set of all $c$-dimensional subspaces that are in $W$ and intersect $U$ trivially is:
\[\gaussmsetminusset{W}{U}{c} = \{ A \le W \mid \dim(A)=c \text{ and } A\cap U = \{0\} \}\text{.}\]
For $w = \dim(W)$ and $u = \dim(U \cap W)$ its cardinality is $\gaussmsetminusnum{w}{u}{c}{q}$ which can be computed:
\[\gaussmsetminusnum{w}{u}{c}{q} = \prod_{i=0}^{c-1} \frac{q^w-q^{u+i}}{q^c-q^i} = q^{uc} \prod_{i=0}^{c-1} \frac{q^{w-u}-q^{i}}{q^c-q^i} = q^{uc} \gaussmnum{w-u}{c}{q}\]
for $0 \le c \le w-u$ and 0 otherwise.
\end{defin}

\section{Bounds on CDCs containing LMRDs}\label{sec:maintheorems}

In general, any $(k-d/2+1)$-dimensional subspace of $V$ is contained in at most one codeword of a $(v,\#C,d;k)_q$ CDC $C$.
If $C$ contains an LMRD $M$, all $(k-d/2+1)$-subspaces in $\Delta$ are covered by codewords in $M$.
More precisely:

\begin{lemma}[{\cite[Lemma~4]{MR3015712}}] \label{lem:etzion_subspace}
Using $2 \le d/2 \le k \le v/2$, each $(k-d/2+1)$-dimensional subspace of $V$, whose nonzero vectors are in $\Delta$, is subspace of exactly one element of a $(v,q^{(v-k)(k-d/2+1)},d;k)_q$ LMRD code.
\end{lemma}
\begin{proof}
The number of $(k-d/2+1)$-dimensional subspaces in $\Delta$ is
\[\#\gaussmsetminusset{V}{\Gamma}{k-d/2+1} = \gaussmsetminusnum{v}{v-k}{k-d/2+1}{q} = q^{(v-k)(k-d/2+1)} \gaussmnum{k}{k-d/2+1}{q}\text{.}\]
The cardinality of an LMRD code is $q^{(v-k)(k-d/2+1)}$, it contains only nonzero vectors from $\Delta$, and, since each $(k-d/2+1)$-dimensional subspace is contained in exactly one codeword, the statement follows.
\end{proof}

\begin{lemma}\label{lem:subspaces_in_delta}
Any subspace $U$ of $V$ contains a $(\dim(U)-\dim(U \cap \Gamma))$-dimensional subspace whose nonzero vectors are in $\Delta$.
\end{lemma}
\begin{proof}
By definition of $\Delta$ all vectors in $U \setminus (U \cap \Gamma)$ are in $\Delta$. Then basis extension yields a desired subspace.
\end{proof}

These two lemmata will now show that the non-LMRD codewords in a CDC which contains an LMRD have to have a large intersection with $\Gamma$, which is of course not true for general CDCs.

\begin{lemma}\label{lem:cdc_partition}
Using $2 \le d/2 \le k \le v/2$, any $(v,\#C,d;k)_q$ CDC $C$ that contains an LMRD code $M$ can be partitioned into
\[C = M \ \dot\cup \ \  \dot\bigcup_{t=d/2}^{k} \  S_t\text{,}\]
where $S_t = \{ U \in C \mid \dim(U \cap \Gamma) = t \}$, and
\[d_S(A\cap \Gamma,B\cap \Gamma) \ge d_S(A,B)-2k+a+b\]
for $A \in S_a$ and $B \in S_b$.
\end{lemma}
\begin{proof}
A subspace $U \in C$ with $\dim(U \cap \Gamma) \le d/2-1$ yields via Lemma~\ref{lem:subspaces_in_delta} an at least $(k-d/2+1)$-dimensional subspace $W$ with nonzero vectors in $\Delta$. Then Lemma~\ref{lem:etzion_subspace} shows that $W_0 \le W$, $\dim(W_0)=k-d/2+1$, is contained in exactly one codeword in $M$, i.e., $U \in M$. Moreover, using the minimum distance, $W_0$ is in at most one element of $C$.

For $A \in S_a$ and $B \in S_b$ we have $\dim(A \cap B \cap \Gamma) \le \dim(A \cap B) = k-d_S(A,B)/2$, hence $d_S(A\cap \Gamma,B\cap \Gamma)=a+b-2\dim(A \cap B \cap \Gamma) \ge d_S(A,B)-2k+a+b$.
\end{proof}

Using this lemma, we can upper bound the size of a $(v,\#C,d;k)_q$ CDC $C$ that contains an LMRD $M$, for $2 \le d/2 \le k \le v/2$, via
\[\#C = \#M + \sum_{t=d/2}^{k} \#S_t = q^{(v-k)(k-d/2+1)} + \sum_{t=d/2}^{k} \#S_t\text{.}\]

The following trick may be observed in \cite[Theorem~3]{ahlswede2009error}.

\begin{lemma}\label{lem:subspaces_upper_bound}
Let $l<2m$ be an integer and $A_i \subseteq \gaussmset{V}{i}$ for $m \le i \le M$ such that $d_S(U,W) \ge \dim(U) + \dim(W) -l$ for $U \ne W \in \bigcup_{i=m}^{M} A_i$. Then
\[ \# \bigcup_{i=m}^{M} A_i \le A_q(v,2m-l;m)\text{.} \]
\end{lemma}
\begin{proof}
For each $m \le i \le M$, we define $B_i=\{ \mathcal{H}_m(U) \mid U \in A_i \}$.
Then the set $C=\bigcup_{i=m}^{M} B_i$ is a $(v,\#\bigcup_{i=m}^{M} A_i,2m-l;m)_q$ CDC.
The cardinality follows from the minimum distance, i.e., for $\tilde{U} \ne \tilde{W} \in C$ such that $U \in A_u$ yielded $\tilde{U}$ and $W \in A_w$ yielded $\tilde{W}$, we have $u+w-l \le d_S(U,W) = u+w-2\dim(U \cap W) \Rightarrow \dim(\tilde{U} \cap \tilde{W}) \le \dim(U \cap W) \le l/2$ and $d_S(\tilde{U},\tilde{W})=2(m-\dim(\tilde{U}\cap \tilde{W})) \ge 2(m-l/2) >0$.
\end{proof}

\begin{proposition}[cf. {\cite[Theorem~10]{MR3015712}}] \label{prop:1}
For $2 \le d/2 \le k \le v/2$ let $C$ be a $(v,\#C,d;k)_q$ CDC that contains an LMRD code where $k < d$. Then
\[\#C \le q^{(v-k)(k-d/2+1)} + A_q(v-k,2(d-k);d/2)\text{.}\]
\end{proposition}
\begin{proof}
Using Lemma~\ref{lem:cdc_partition}, we only have to upper bound the size of $\dot\bigcup_{t=d/2}^{k} S_t$.
Applying Lemma~\ref{lem:subspaces_upper_bound} with $A_i = \{U \cap \Gamma \mid U \in S_i\} \subseteq \gaussmset{\Gamma}{i}$, $m=d/2$, $M=k$, and $l=2k-d$ (cf. Lemma~\ref{lem:cdc_partition}) is possible since $0 < 2m-l = 2(d-k) \Leftrightarrow k < d$.
\end{proof}

The special case of $d=2(k-1)$ and $k \ge 3$ was already proved in \cite[Theorem~10]{MR3015712}.

Next, we generalize \cite[Theorem~11]{MR3015712} and need therefore two technical lemmata.

\begin{lemma}\label{lem:something_le_zero}
Let $c,k,q,t,t_0,y$ be integers where $q$ is a prime power, $y \ne 0$, and $c \le k-t$ as well as $t_0 \le t$.
Then
\[\gaussmsetminusnum{k}{t_0}{c}{q} \gaussmnum{t_0}{y}{q} \le \gaussmsetminusnum{k}{t}{c}{q} \gaussmnum{t}{y}{q}\text{.}\]
\end{lemma}
\begin{proof}
Since $t_0 = t$, $c<0$, $y<0$, and $t_0 < y$ as well as $c=0$ are obvious, we assume $1 \le c$ and $1 \le y \le t_0 < t$.
\begin{align*}
&
\frac{\gaussmsetminusnum{k}{t_0}{c}{q} \gaussmnum{t_0}{y}{q}}{\gaussmsetminusnum{k}{t}{c}{q} \gaussmnum{t}{y}{q}} q^{c(t-t_0)}
=
\frac{\gaussmnum{k-t_0}{c}{q} \gaussmnum{t_0}{y}{q}}{\gaussmnum{k-t}{c}{q} \gaussmnum{t}{y}{q}}
=
\frac{[k-t_0]_q! [t_0]_q! [k-t-c]_q! [t-y]_q!}{[k-t]_q! [t]_q! [k-t_0-c]_q! [t_0-y]_q!}
\\
&
=
\prod_{i=t_0+1}^{t}\frac{[k-i-c]_q [i-y]_q}{[k-i]_q [i]_q}
\le
\prod_{i=t_0+1}^{t}q^{-c}q^{-y}
=
q^{-(c+y)(t-t_0)}
\end{align*}
The exponent is negative and therefore we have $\le q^{c(t-t_0)}$.

\end{proof}

Note that the restriction $t_0 \le t$ is the reason for the fixation of $t_0=d/2$ before Proposition~\ref{prop:2}.

\begin{lemma} \label{lem:NtY}
Using the notation of Lemma~\ref{lem:cdc_partition}, let $c$, $t$, and $y$ be integers with $0 \le y \le k$, $d/2 \le t \le k$, and $k-d/2+1 \le c+y$.
Let $N_{t,Y} = \{U \in S_t \mid Y \le U\}$ for each $Y \in \gaussmset{\Gamma}{y}$ with $0 \le y \le k$ and $d/2 \le t \le k$.\footnote{Note that we deliberately use $t < y \le k$ with $N_{t,Y} = \emptyset$.}
Then we have:
\[\sum_{Y \in \gaussmset{\Gamma}{y}} \#N_{t,Y} = \#S_t \cdot \gaussmnum{t}{y}{q}\text{.}\]
Moreover for all $Y \in \gaussmset{\Gamma}{y}$ we have:
\[\sum_{t=d/2}^{k-c} \#N_{t,Y} \cdot \gaussmsetminusnum{k}{t}{c}{q} \le \gaussmsetminusnum{v}{v-k}{c}{q}\text{.}\]
\end{lemma}
\begin{proof}
The equation follows from double-counting the set $\{(Y,U) \in \gaussmset{\Gamma}{y} \times S_t \mid Y \le U\}$.

For the inequality, we have $0$ on the left hand side if $c<0$ or $k-d/2<c$, i.e., we assume $0 \le c \le k-d/2$.
The statement follows from counting
\[\dot\bigcup_{t=d/2}^k\dot\bigcup_{U \in N_{t,Y}}\gaussmsetminusset{U}{\Gamma}{c}\subseteq\gaussmsetminusset{V}{\Gamma}{c}\text{.}\]
The left hand side is disjoint because for fixed $Y$ there is, using $\dim(\langle Y, R \rangle) = y+c \ge k-d/2+1$, at most one element $W \in C$ with $\langle Y, R \rangle \le W$, where $R \in \gaussmsetminusset{V}{\Gamma}{c}$.

Furthermore $\gaussmsetminusset{U}{\Gamma}{c} = \emptyset$ for $k-c<t$ and $U \in N_{t,Y}$.
\end{proof}

In particular, we have for all integers $c$, $t_0$, and $y$ with $0 \le y \le k$, $k-d/2+1 \le c+y$, $Y \in \gaussmset{\Gamma}{y}$, and $d/2 \le t_0 \le k$, as well as $0 \le c \le k-t_0$:
\[\#N_{t_0,Y} \le \frac{\gaussmsetminusnum{v}{v-k}{c}{q} -\sum_{t=d/2, t \ne t_0}^{k-c} \#N_{t,Y} \cdot \gaussmsetminusnum{k}{t}{c}{q}}{\gaussmsetminusnum{k}{t_0}{c}{q}}\text{.}\]

In the successive discussion, we fix $t_0=d/2$ (cf. Lemma~\ref{lem:something_le_zero}), to ease the notation significantly while maintaining the same level of detail: The second summand of the last part of the proof of the next proposition would not vanish for other $t_0$.

\begin{proposition}[cf. {\cite[Theorem~11]{MR3015712}}] \label{prop:2}
For $2 \le d/2 \le k \le v/2$ let $C$ be a $(v,\#C,d;k)_q$ CDC that contains an LMRD code for integers $c$ and $y$ such that
$1 \le y \le d/2$,
$1 \le c \le \min\{k-d/2,d/2\}$, and
$k-d/2+1 \le c+y$.
Then
\begin{align*}
\#C
\le
q^{(v-k)(k-d/2+1)}
+
\frac{\gaussmnum{v-k}{y}{q}\gaussmnum{k}{c}{q}}{\gaussmnum{k-d/2}{c}{q}\gaussmnum{d/2}{y}{q}}
q^{c(v-k-d/2)}
+
A_q(v-k,d-2(c-1);k-c+1)
\text{.}
\end{align*}
\end{proposition}
\begin{proof}
Using Lemma~\ref{lem:cdc_partition} we only have to upper bound $\sum_{t=d/2+1}^{k} \#S_t + \#S_{d/2}$. Applying Lemma~\ref{lem:NtY}, we get:

\begin{align*}
\#S_{d/2}
&
=
\frac{\sum_{Y \in \gaussmset{\Gamma}{y}} \#N_{d/2,Y}}{\gaussmnum{d/2}{y}{q}}
\le
\sum_{Y \in \gaussmset{\Gamma}{y}}
\frac{\gaussmsetminusnum{v}{v-k}{c}{q} -\sum_{t=d/2+1}^{k-c} \#N_{t,Y} \gaussmsetminusnum{k}{t}{c}{q}}{\gaussmsetminusnum{k}{d/2}{c}{q} \gaussmnum{d/2}{y}{q}}
\\
&
=
\frac{\gaussmnum{v-k}{y}{q} \gaussmsetminusnum{v}{v-k}{c}{q} -\sum_{t=d/2+1}^{k-c} \#S_t \gaussmnum{t}{y}{q} \gaussmsetminusnum{k}{t}{c}{q}}{\gaussmsetminusnum{k}{d/2}{c}{q} \gaussmnum{d/2}{y}{q}}
\text{.}
\end{align*}

Hence

\begin{align*}
&
\sum_{t=d/2+1}^{k} \#S_t + \#S_{d/2}
\le
\frac{\gaussmnum{v-k}{y}{q} \gaussmsetminusnum{v}{v-k}{c}{q}}{\gaussmsetminusnum{k}{d/2}{c}{q} \gaussmnum{d/2}{y}{q}}
\\
&
+
\frac{\sum_{t=d/2+1}^{k-c} \#S_t \left(\gaussmsetminusnum{k}{d/2}{c}{q} \gaussmnum{d/2}{y}{q}-\gaussmnum{t}{y}{q} \gaussmsetminusnum{k}{t}{c}{q}\right)}{\gaussmsetminusnum{k}{d/2}{c}{q} \gaussmnum{d/2}{y}{q}}
+
\sum_{t=k-c+1}^{k} \#S_t
\text{.}
\end{align*}

Now we apply Lemma~\ref{lem:something_le_zero} for $t_0=d/2$ and $d/2+1 \le t \le k-c$, and thereby upper bound the second summand with zero.

The last summand can be upper bounded by utilizing again Lemma~\ref{lem:subspaces_upper_bound} with $A_i = \{U \cap \Gamma \mid U \in S_i\} \subseteq \gaussmset{\Gamma}{i}$, $m=k-c+1$, $M=k$, which is possible since $1 \le c$, and $l=2k-d$ (cf. Lemma~\ref{lem:cdc_partition}), using $0 < 2m-l = d-2(c-1) \Leftrightarrow c \le d/2$. This upper bounds the last summand with $A_q(v-k,d-2(c-1);k-c+1)$.

\end{proof}

The special case of $d=k$ even, $c=1$, $y=d/2$ was already proved in \cite[Theorem~11]{MR3015712}.

\section{Comparison of the bounds}\label{sec:comparisonbounds}

First, small values of $y(c)$ are better.

\begin{remark}\label{rem:optimal_choice_y}
Using $2 \le d/2 \le k \le v/2$, the function $f(y)=\gaussmnum{v-k}{y}{q} / \gaussmnum{d/2}{y}{q} = \prod_{i=0}^{y-1} \frac{q^{v-k}-q^i}{q^{d/2}-q^i}$ is monotonically increasing for $1 \le y \le d/2$.
Hence for fixed $c$ the optimal choice for $y$ is $\max\{1,k-d/2+1-c\}$, which implies $\max\{1,k-d+1\} \le c \le \min\{k-d/2,d/2\}$.
Note that such a $c$ exists iff $d/2 < k < 3d/2$.
\end{remark}

Second, small values of $c$ are better for the third summand of Proposition~\ref{prop:2}.

\begin{lemma}\label{lem:third_summand}
For a prime power $q$ and integers $v \ge 0$ and $k \ne 0$, we have
\[A_q(v,d;k) \le A_q(v,d-2;k-1)\text{.}\]
\end{lemma}
\begin{proof}
For $k<0$, $v<k$, $2k < d$, $v \le 1$, or $d \le 2$ the statement is obvious.
For odd $d$ we can use $\tilde{d}=d+1$ due to $A_q(v,d;k)=A_q(v,d+1;k)$.
We estimate the left hand side with the Singleton bound and the right hand side with the size of an LMRD code. Since both bounds depend on whether $k \le v/2$, we have these three cases:

If $k \le v/2$ then
\begin{align*}
A_q(v,d;k)
\le
\gaussmnum{v-d/2+1}{v-k}{q}
\le
\mu(q) q^{(v-k)(k-d/2+1)}
\le
q^{(v-k+1)(k-d/2+1)}
\le
A_q(v,d-2;k-1)
\end{align*}
which is true for $q \ge 3$, since $\mu(q) \le q \le q^{k-d/2+1}$, and $q=2$ with $2 \le k-d/2+1$. For $q=2$ and $d=2k$, the Singleton bound is $\gaussmnum{v-k+1}{1}{2}=2^{v-k+1}-1$ yielding the result.

If $v/2 \le k-1$ then
\begin{align*}
&
A_q(v,d;k)
=
A_q(v,d;v-k)
\le
\gaussmnum{v-d/2+1}{k}{q}
\le
\mu(q) q^{k(v-k-d/2+1)}
\\
&
\le
q^{(k-1)(v-k-d/2+3)}
\le
A_q(v,d-2;v-k+1)
=
A_q(v,d-2;k-1)
\end{align*}
which is true for $1 \le k$, since $\mu(q) \le q^2 \le q^{3k-3-v+d/2}$, i.e., $v+5 \le 2k+3 \le 2k+1+d/2 \le 3k+d/2$.

If $v$ is odd and $k = (v+1)/2$ then
\begin{align*}
&
A_q(v,d;k)
=
A_q(v,d;(v+1)/2)
=
A_q(v,d;(v-1)/2)
\le
\gaussmnum{v-d/2+1}{(v+1)/2}{q}
\\
&
\le
\mu(q) q^{((v-1)/2-d/2+1)(v+1)/2}
\le
q^{((v-1)/2-d/2+2)(v+1)/2}
\\
&
\le
A_q(v,d-2;(v-1)/2)
=
A_q(v,d-2;k-1)
\end{align*}
which is true for $3 \le v$ since $\mu(q) \le q^2 \le q^{(v+1)/2}$.
\end{proof}

The second summand of Proposition~\ref{prop:2} is monotonically increasing in $c$ and therefore smaller values of $c$ are also better for this term.

\begin{lemma}\label{lem:second_summand}
For integers $c$, $d$, $k$, $q$, $v$, and $y(c)$ such that $q$ is a prime power, $2 \le d/2 \le k \le v/2$, $0 \le c \le k-d/2-1$, $0 \le y(c) \le d/2$, $2 \le d/2$, and $2 \le v-k$ let
\[f(c)=\frac{\gaussmnum{v-k}{y(c)}{q}\gaussmnum{k}{c}{q}}{\gaussmnum{k-d/2}{c}{q}\gaussmnum{d/2}{y(c)}{q}} q^{c(v-k-d/2)}\text{.}\]

If $y(c+1) = y(c)$ or $y(c+1) = y(c)-1 \ge 0$, then $f(c) \le f(c+1)$.
\end{lemma}
\begin{proof}
The term
\[
\lambda=\frac{[d/2-y(c)]_q!}{[d/2-y(c+1)]_q!} \cdot \frac{[v-k-y(c+1)]_q!}{[v-k-y(c)]_q!}
\]
is $1$ if $y(c+1)=y(c)$ and
\[\frac{[v-k-y(c)+1]_q}{[d/2-y(c)+1]_q} \le \mu(q) q^{v-k-d/2}\]
if $y(c+1)=y(c)-1$.
Using the $q$-factorial version of the $q$-binomial coefficient one gets:
\begin{align*}
\frac{f(c)}{f(c+1)}
&=
\frac{q^{k-d/2-c}-1}{q^{k-c}-1} \cdot q^{-(v-k-d/2)} \cdot \lambda
\\
&
\le
q^{-d/2} \cdot q^{-(v-k-d/2)} \cdot \lambda
\\
&
\le
\begin{cases}
q^{-(v-k)} &\text{if } y(c+1)=y(c) \\
\mu(q) q^{-(d/2)} \le q^{2-(d/2)} &\text{else} \\
\end{cases}
\\
&\le
1
\end{align*}

\end{proof}

The second summand in Proposition~\ref{prop:2} is already larger than the second summand in Proposition~\ref{prop:1}.

\begin{lemma}\label{lem:compare_prop1_prop2}
Let $d$, $k$, $q$, and $v$ be integers such that $q$ is a prime power,
$2 \le d/2 < k < d \le 2v/3$,
$k+d/2 \le v$,
$c=1$, and
$y=k-d/2$.
Then we have
\[
A_q(v-k,2(d-k);d/2)
\le
\frac{\gaussmnum{v-k}{y}{q}\gaussmnum{k}{c}{q}}{\gaussmnum{k-d/2}{c}{q}\gaussmnum{d/2}{y}{q}} q^{c(v-k-d/2)}
\text{.}
\]
\end{lemma}
\begin{proof}
Applying the Singleton bound yields
\begin{align*}
&
A_q(v-k,2(d-k);d/2)
\le
\gaussmnum{v-d+1}{v-k-d/2}{q}
\le
\mu(q) q^{(v-k-d/2)(k-d/2+1)}
\le
q^{(v-k-d/2)(k-d/2+1)+d/2}
\\
&
=
q^{(v-k-d/2)(k-d/2)} \cdot q^{d/2} \cdot q^{v-k-d/2}
\le
\prod_{i=1}^{y} \frac{[v-k-y+i]_q}{[d/2-y+i]_q} \cdot \frac{[k]_q}{[k-d/2]_q} \cdot q^{v-k-d/2}
\\
&
=
\frac{[v-k]_q![k]_q![k-d/2-c]_q![d/2-y]_q!}{[v-k-y]_q![k-c]_q![k-d/2]_q![d/2]_q!} \cdot q^{v-k-d/2}
=
\frac{\gaussmnum{v-k}{y}{q}\gaussmnum{k}{c}{q}}{\gaussmnum{k-d/2}{c}{q}\gaussmnum{d/2}{y}{q}} q^{c(v-k-d/2)}
\end{align*}

\end{proof}

\section{Proof of Proposition~\ref{prop:0} and final remarks}\label{sec:proof_prop0}

Here we state the proof of Proposition~\ref{prop:0}, which can be found in the introduction.

First, we discuss the optimal choice of $y$ and $c$.
Remark~\ref{rem:optimal_choice_y} shows that the optimal choice for $y$ is $\max\{1,k-d/2+1-c\}$.
Then for $\max\{1,k-d+1\} \le c \le \min\{k-d/2,d/2\}$ we compare the second summand and the third summand of the statement in Proposition~\ref{prop:2} separately.
The third summand, i.e., $A_q(v-k,d-2(c-1);k-c+1)$ is monotonically decreasing in $c$ as seen in Lemma~\ref{lem:third_summand}.
The second summand, i.e.,
\[\frac{\gaussmnum{v-k}{y(c)}{q}\gaussmnum{k}{c}{q}}{\gaussmnum{k-d/2}{c}{q}\gaussmnum{d/2}{y(c)}{q}} q^{c(v-k-d/2)}\]
is also monotonically decreasing in $c$ by Lemma~\ref{lem:second_summand}.
Hence the smallest $c$ yields the smallest upper bound and therefore is $\max\{1,k-d+1\}$ the optimal choice for $c$.

Second, we compare the bound in Proposition~\ref{prop:2} to the bound in Proposition~\ref{prop:1} where both bounds are applicable, i.e., $d/2 < k < d$.
The second summand of Proposition~\ref{prop:2}, utilizing the optimal choice of $y$ and $c$, is already larger than the second summand of Proposition~\ref{prop:1} by Lemma~\ref{lem:compare_prop1_prop2}.

Hence we only apply Proposition~\ref{prop:2} for $d \le k < 3d/2$ and in particular $d \le k$ shows $c=k-d+1 \ge 1$ and $y=d/2 \ge 2$.

The restriction $v<3d/2$ is equivalent to $2(v-k-d/2) < 2(d-k)$, i.e., any two codewords $U\ne W$ in an orthogonal $(v-k,\#C,2(d-k);d/2)_q$ code have $d_S(U,W) \le 2(v-k-d/2) < 2(d-k)$, hence $\#C \le 1$.
Moreover a code attaining this bound can be constructed by extending an $(v,\#M,d;k)_q$ LMRD with the codeword $Z=\tau^{-1}(0_{v-k} \mid I_k)$ since $2k \le v$ implies that $Z$ intersects each other codeword trivially.

In addition to the trivial cases in the last proof, the second summand in Proposition~\ref{prop:1} is known in further cases:

\begin{remark}\label{rem:partial_spreads_and_spreads}
For $2 \le d/2 \le k \le v/2$ as well as $k<d \le 2v/3$:

If $d=2k$, then $A_q(v-k,2(d-k);d/2)$ corresponds to a partial spread and if in addition $r \equiv v \mod{k}$, $0 \le r <k$, and $[r]_q<k$ then $A_q(v-k,2(d-k);d/2)=\frac{q^{v-k}-q^{k+r}}{q^{k}-1}+1$~\cite{MR3682737}.
Hence the bound in Proposition~\ref{prop:0} is $\#C \le \sum_{i=0}^{l-1} q^{k(l-i)+r} +1 = A_q(v,d;k)$, using $v-k-r = lk$.
An optimal CDC containing an LMRD can be constructed with the Echelon-Ferrers construction and the pivots $p_i=(0_{ik} 1_{k} 0_{v-(i+1)k})$ for $i=0,\ldots, l$.

If $v=3d/2$, then $A_q(v-k,2(d-k);d/2)$ corresponds to an orthogonal partial spread and if in addition $d-k \mid d/2$, it corresponds to a spread of size $(q^{3d/2-k}-1)/(q^{d-k}-1)$.
\end{remark}

\begin{lemma}
There is, for integral $l \ge 1$ and prime power $q$, a $(6l,q^{3l(l+1)}+q^{2l}+q^l+1,4l;3l)_q$ CDC $C$ that contains an LMRD. This cardinality achieves the bound of Proposition~\ref{prop:0}.
\end{lemma}
\begin{proof}
The bound of Proposition~\ref{prop:0} can be computed via Remark~\ref{rem:partial_spreads_and_spreads}.

$C$ is constructed with the Echelon-Ferrers construction and these pivot vectors:

$( 1_{l} 1_{l} 1_{l} 0_{l} 0_{l} 0_{l} )$ (i.e., an LMRD of size $q^{3l(l+1)}$)

$( 1_{l} 0_{l} 0_{l} 1_{l} 1_{l} 0_{l} )$

$( 0_{l} 1_{l} 0_{l} 1_{l} 0_{l} 1_{l} )$

$( 0_{l} 0_{l} 1_{l} 0_{l} 1_{l} 1_{l} )$ (i.e., a subcode with $1$ element )

Note that the Hamming distances between these four constant weight codewords is always $4l$ which implies the subspace distance of at least $4l$ by Lemma~\ref{lem:dh_ds}.
The size of the subcode, corresponding to the second constant weight codeword, is $q^{2l}$ and can be constructed with Lemma~\ref{lem:EFMRD_special_case} and two $[l \times 2l, q^{2l}, l]_q$ MRDs.
The third constant weight codeword gives rise to $q^{l}$ codewords of $C$ using the same technique and two $[l \times l, q^{l}, l]_q$ MRDs.
\end{proof}

Previously, only the optimality for $l=1$ was known \cite[Theorem~10]{MR3015712}.

Another series of LMRD bound achieving parameters is:

\begin{lemma}
There is, for integral $l \ge 1$ and prime power $q$, a $(6+3l,q^{6+4l}+q^{2+l}+1,4+2l;3+l)_q$ CDC $C$ that contains an LMRD. This cardinality achieves the bound of Proposition~\ref{prop:0}.
\end{lemma}
\begin{proof}
First, the bound is given by $\#C \le q^{6+4l}+A_q(3+2l,2+2l;2+l)$. The second summand is, due to orthogonal codes and $3+2l \equiv 1 \mod{(1+l)}$ for $l \ge 1$, known \cite{MR0404010} and equal to $q^{2+l}+1$.

Second, $C$ can be constructed with the Echelon-Ferrers construction and these pivot vectors:

$(1_1 1_{1+l} 1_1 0_{1+l} 0_1 0_{1+l})$ (i.e., an LMRD of size $q^{6+4l}$)

$(1_1 0_{1+l} 0_1 1_{1+l} 1_1 0_{1+l})$

$(0_1 0_{1+l} 1_1 0_{1+l} 1_1 1_{1+l})$ (i.e., a subcode with $1$ element )

Note that the Hamming distances between these three constant weight codewords is always $4+2l$ which implies the subspace distance of at least $4+2l$ by Lemma~\ref{lem:dh_ds}.
The size of the subcode, corresponding to the second constant weight codeword, is $q^{2+l}$ and can be constructed with Lemma~\ref{lem:EFMRD_special_case}, a $[1 \times (2+l), q^{2+l}, 1]_q$ MRD and a $[(2+l) \times (1+l), q^{2+l}, 1+l]_q$ MRD.
\end{proof}

For all prime powers $q$ and integral $l \ge 1$, this bound was previously known \cite[Theorem~10]{MR3015712} and is listed here for completeness.

\section{Improved code sizes} \label{sec:code_improvements}

Since Lemma~\ref{lem:cdc_partition} states that any $(v,\#C,d;k)_q$ CDC that contains an LMRD $M$ can be partitioned into $C= M \dot\cup S_{d/2} \dot\cup \ldots \dot\cup S_{k}$, we know that any codeword in $C \setminus M$ has an at least $d/2$-dimensional intersection with $\Gamma$.
Hence we describe a promising approach to find large codes $C$ by considering $E \subseteq \gaussmset{\Gamma}{d/2}$.
If $k<d$, i.e., $k-d/2+1 \le d/2$, then any codeword in $C \setminus M$ contains different elements in $E$.
Moreover, Lemma~\ref{lem:cdc_partition} also states, that the minimum distance of $E$ has to be at least $2(d-k)$, cf. Proposition~\ref{prop:1}, with other words, $E$ is a $(v-k,\#E,2(d-k);d/2)_q$ CDC.
Therefore it is natural to consider already large CDCs, which are for example listed in~\cite{HKKW2016Tables} and try to extend them.
On the other hand, a given $(v',N',d';k')_q$ CDC, where $2 \le d'$, can be used to build a $(v'+2k'-d'/2,N,2k';2k'-d'/2)_q$ CDC, $N \le N'$, that is compatible to any LMRD that respects these parameters.

Moreover, if $k<d$, then a $(v,\#C,d;k)_q$ CDC $C$ that contains an LMRD $M$ implies a $(v-k,\#C-\#M,2(d-k);d/2)_q$ CDC $C'=\{\mathcal{H}_{d/2}(U \cap \Gamma) \mid U \in C \setminus M \}$, which in turn shows that generating a large $C$ is at least as difficult as generating $C'$.

Next, the number of subspaces in $C\setminus M$ having a large intersection with $\Gamma$ is limited by $\# S_t \le A_q(v-k,d-2(k-t);t)$ for $\max\{d/2,k-d/2+1\} \le t \le k$ as an application of Lemma~\ref{lem:subspaces_upper_bound}, $m=M=t$, $l=2k-d$, $A_t=\{U \cap \Gamma \mid U \in S_t\} \subseteq \gaussmset{\Gamma}{t}$, with $\#A_t=\#S_t$, due to the minimum distance $d_S(U \cap \Gamma, W \cap \Gamma) \ge d_S(U, W) -2k+2t \ge d-2k+2t >0$, shows.

\begin{algorithm}
\caption{Random search strategy for extending an arbitrary LMRD}
\label{alg:search}
\begin{algorithmic}[1]
\Require{$E$ is a $(v-k,\#E,2(d-k);d/2)_q$ CDC embedded in $\Gamma$, $1 \le n_{\max}$, and $1 \le r_{\max}$ integers}
\Procedure{Search}{$E,n_{\max},r_{\max}$}
\State $T \gets \tau\left(\gaussmset{\mathbb{F}_q^{v-d/2}}{k-d/2}\right)$
\State $C_{\max} \gets \{\}$
\For{$n \in \{1, \ldots, n_{\max}\}$}
\State $C \gets \{\}$
\For{$U \in E$}
\State $A \in \gaussmset{V}{v-d/2}$
\Comment{such that $A \oplus U = V$}
\State $M \gets \tau(A)$
\State $\sigma \gets \operatorname{random}(\mathcal{S}_{\#T})$
\For{$r \in \{1, \ldots, \min\{r_{\max},\#T\}\}$}
\State $W \gets U \oplus \tau^{-1}(T_{\sigma(r)} \cdot M)$
\For{$Z \in C$}
\If{$\dim(Z \cap W) > k-d/2$}
\State \textbf{continue} $r$
\EndIf
\EndFor
\State $C \gets C \cup W$
\If{$k<d$}
\State \textbf{continue} $U$
\EndIf
\EndFor
\EndFor
\If{$\#C > \#C_{\max}$}
\State $C_{\max} \gets C$
\EndIf
\EndFor
\State \textbf{return} $C_{\max}$
\EndProcedure
\end{algorithmic}
\end{algorithm}

For a given subcode $E$, Algorithm~\ref{alg:search} shows our applied search strategy.
Note, that the argument $r_{\max}$ controls the level of detail of each of the independent $n_{\max}$ runs.
Note further, that we do not precompute the set of extensions for each subspace in $E$ although it may be useful to save computation time if $r_{\max}$ is large compared to the size of the set of extensions, i.e. $\gaussmnum{v-d/2}{k-d/2}{q}$, and $n_{\max}$ is at least two.

Table~\ref{tab:new_code_sizes} lists improved sizes of CDCs for small fixed parameters $q$, $v$, $d$, and $k$. The size of the LMRD with this parameters is $\#M$ and the successive columns show only the extended cardinality to the corresponding LMRD size. Therefore LMRD-B is the size of the LMRD bound, PKLB is the previously best known lower bound, $E$ is the used subcode up to embedding in $\Gamma$, and BKLB is the current best known lower bound.
The codes can be downloaded from \url{http://subspacecodes.uni-bayreuth.de}, see also~\cite{HKKW2016Tables}.

\begin{table}
\setlength{\tabcolsep}{4pt}

\caption{New lower bounds on some CDC parameters}\label{tab:new_code_sizes}
\centering
\begin{tabular}{llll|lllll}
$q$ & $v$ & $d$ & $k$ & $\#M$ & \twolines{LMRD-B}{$-\#M$} & \twolines{PBKLB}{$-\#M$} & $E$ & \twolines{BKLB}{$-\#M$} \\

\hline

$2$ & $10$ & $6$ & $5$ & $2^{15}$ & $155$ & \twolines{$122$}{\cite[Ex.~4]{MR2801585}} & $\gaussmset{\Gamma}{3}$ & $155$ \\

$2$ & $11$ & $6$ & $4$ & $2^{14}$ & \twolines{$A_2(7,4;3)$}{$\le 381$} & \twolines{$285$}{\cite{MR2589964,HKKW2016Tables}} & \twolines{$(7,333,4;3)_2$}{\cite{HKKW2016Tables}} & $333$ \\

$2$ & $11$ & $6$ & $5$ & $2^{18}$ & $1395$ & \twolines{$852$}{\cite{MR2589964,HKKW2016Tables}} & $\gaussmset{\Gamma}{3}$ & $1334$ \\

$2$ & $12$ & $6$ & $4$ & $2^{16}$ & \twolines{$A_2(8,4;3)$}{$\le 1493$} & \twolines{$1144$}{\cite{MR2589964,HKKW2016Tables}} & \twolines{$(8,1326,4;3)_2$}{\cite{new_lower_bounds_cdc}} & $1303$ \\

$2$ & $12$ & $6$ & $5$ & $2^{21}$ & $11811$ & \twolines{$7232$}{\cite{MR2589964,HKKW2016Tables}} & $\gaussmset{\Gamma}{3}$ & $7925$ \\

$2$ & $13$ & $6$ & $4$ & $2^{18}$ & \twolines{$A_2(9,4;3)$}{$\le 6205$} & \twolines{$4747$}{\cite{MR3367813}} & \twolines{$(9,5986,4;3)_2$}{\cite{new_lower_bounds_cdc}} & $5753$ \\
\end{tabular}
\end{table}

Note that a further improvement of the second code, i.e. $(q,v,d,k)=(2,11,6,4)$, would imply a $(7,\#E,4;3)_2$ CDC $E$ with $333 < \#E$.

The situation of the first code, i.e., $(q,v,d,k)=(2,10,6,5)$, is a special case, since $\#S_3 \le 155$, $\#S_4 \le 1$, and $S_5 \subseteq \{\Gamma\}$.

If $\#S_5=1$, then $\#S_3=\#S_4=0$ because any subspace $U \in S_3 \cup S_4$ has $d_S(U,\Gamma) \le 4$, hence we set $S_5 = \emptyset$.

If $\#S_4=1$, then $\#S_3 \le 140$ because for $U \in S_4$ we have $\#\{ W \in S_3 \mid \dim((U \cap \Gamma) \cap (W \cap \Gamma)) = 3 \} = \gaussmnum{4}{3}{2}=15$, i.e., the elements in this set have $d_S(U,W)=2(5-3)=4$ and, aiming for large code sizes, we set $S_4 = \emptyset$.

Therefore, a code with these parameters that contain an LMRD and achieves the LMRD bound has to contain a subcode $S_3$  of cardinality $155$, i.e., all subspaces $\gaussmset{\Gamma}{3}$ have to be extended with subspaces in $\gaussmsetminusset{V}{\Gamma}{2}$ such that the minimum distance constraint is fulfilled.
Note that the subspace distance of any codeword $U \in M$ and $W \in S_3 = C \setminus M$ is at least $6$ and therefore solely the minimum distance of $S_3$ is in question.

There are, for each subspace in $\gaussmset{\Gamma}{3}$, $\gaussmnum{10-3}{5-3}{2}=2667$ extensions to $5$ dimensions, of which $2480$ intersects $\Gamma$ $3$-dimensional.

Hence by prescribing the following subgroup of order $31$ of the stabilizer of $\Gamma$, i.e., the cyclic group generated by a block diagonal matrix consisting of twice the same generator of a Singer cycle in $\Gamma$,
\[
G=
\left\langle\left(
\begin{smallmatrix}
0&0&0&0&1\\
1&0&0&0&0\\
0&1&0&0&1\\
0&0&1&0&0\\
0&0&0&1&0\\
&&&&& 0&0&0&0&1\\
&&&&& 1&0&0&0&0\\
&&&&& 0&1&0&0&1\\
&&&&& 0&0&1&0&0\\
&&&&& 0&0&0&1&0\\
\end{smallmatrix}
\right)\right\rangle,
\]
we partition the set $\{U \in \gaussmset{\F_2^{10}}{5} \mid \dim(U \cap \Gamma)=3\}$ of size $2480 \cdot 155 = 384400$ into $12400$ orbits of length $31$ under the action of $G$. $3100$ of these orbits contain a pair of subspaces that has an intersection of at least dimension $3$ and hence these orbits cannot be subset of a $(10,N,6;5)_2$ CDC. The remaining $9300$ orbits are then considered as vertices of a graph in which two vertices $O_1 \ne O_2$ share an edge iff $\dim(U \cap W) \le 2$ for all $U \in O_1$ and $W \in O_2$. Clearly the clique number is upper bounded by $5=\#\gaussmset{\Gamma}{3}/\#G$ since each $3$-dimensional subspaces in $\Gamma$ may be contained at most once without violating the minimum distance. A greedy clique search provides a clique of size $5$, with other words these five orbits are an extension of any $(10,2^{15},6;5)_2$ LMRD of size $155$ achieving the LMRD bound of Proposition~\ref{prop:0}. Representatives in RREF of these five orbits are
\begin{align*}
&
\left(\begin{smallmatrix}
1&0&0&0&0&0&0&0&0&0\\
0&0&1&1&1&0&0&0&1&0\\
&&&&& 1&0&0&0&0\\
&&&&& 0&1&1&0&0\\
&&&&& 0&0&0&0&1\\
\end{smallmatrix}\right),
\left(\begin{smallmatrix}
1&0&0&0&0&0&0&0&0&0\\
0&1&0&0&0&0&1&0&0&0\\
&&&&& 1&0&0&0&0\\
&&&&& 0&0&1&0&0\\
&&&&& 0&0&0&1&0\\
\end{smallmatrix}\right),
\left(\begin{smallmatrix}
0&1&0&0&0&0&0&0&0&0\\
0&0&1&0&0&0&0&0&0&1\\
&&&&& 1&0&0&0&1\\
&&&&& 0&1&0&1&0\\
&&&&& 0&0&1&0&1\\
\end{smallmatrix}\right),
\\
&
\left(\begin{smallmatrix}
1&0&0&0&0&0&0&0&0&0\\
0&0&0&1&1&0&0&0&1&0\\
&&&&& 1&0&0&0&1\\
&&&&& 0&1&0&1&1\\
&&&&& 0&0&1&0&1\\
\end{smallmatrix}\right),
\left(\begin{smallmatrix}
1&0&0&0&0&0&0&0&0&0\\
0&0&0&0&1&0&0&0&1&1\\
&&&&& 1&0&0&0&0\\
&&&&& 0&1&0&0&1\\
&&&&& 0&0&1&1&1\\
\end{smallmatrix}\right),
\end{align*}
in which the omitted parts are zeros since the corresponding rows are RREF matrices of the $3$-dimensional intersection with $\Gamma$.

\section{Conclusion}\label{sec:conclusion}

In this paper we generalize the bounds for the cardinality of constant dimension codes that contain a lifted maximum rank distance code, first studied in \cite[Theorems~10 and 11]{MR3015712}, to a larger set of parameters.
Now we have bounds for the size of CDCs containing LMRDs as subset which are not applicable for general CDCs iff $k < 3d/2$.
It remains an open question if there are LMRD bounds for $3d/2 \le k$.
Furthermore the proofs of these bounds provide new insights in the structure of extensions of lifted maximum rank distance codes and give rise to six new largest CDCs.


\end{document}